\documentclass[12pt, reqno]{amsart}

\usepackage[margin=1in]{geometry}
\usepackage[%pdftex,
  bookmarks=true]{hyperref}
\usepackage[utf8]{inputenc}
\usepackage[T1]{fontenc}
\usepackage[english]{babel}
\usepackage{csquotes}
\usepackage[final]{microtype}
\usepackage{lmodern}
\usepackage{amsthm}
\usepackage{amssymb}
\usepackage{mathrsfs}
\usepackage{enumerate}
\usepackage{tikz-cd} % commutative diagrams
\usepackage{tikz}
\usetikzlibrary{arrows,calc,matrix}
\usepackage{stmaryrd}
\numberwithin{equation}{section}
\usepackage{silence}
\makeindex

\newtheorem{theorem}{Theorem}[section]

\newtheorem{lemma}[theorem]{Lemma}
\newtheorem{corollary}[theorem]{Corollary}

\theoremstyle{definition}

\newtheorem{remark}[theorem]{Remark}
\newtheorem{definition}[theorem]{Definition}

\newcommand{\et}{\acute{e}t} %etale fundamental group
\newcommand{\fppf}{\mathrm{fppf}} %fppf fundamental group
\newcommand{\Gal}{\mathrm{Gal}} %Galois group
\newcommand{\Spec}{\mathrm{Spec }} %Spectrum of a Ring
\newcommand{\Sel}{\mathrm{Sel}_p} %p-Selmer Group
\newcommand{\Conf}{\mathrm{Conf}^n} %configuration space
\newcommand{\neron}{N\'{e}ron } %neron model
\newcommand{\Z}{\mathbb{Z}}
\newcommand{\F}{\mathbb{F}}
\newcommand{\A}{\mathbb{A}}
\newcommand{\Tr}{\mathrm{Tr}}
\newcommand{\Frob}{\mathrm{Frob}_q} %Frobenius for Niudun
\newcommand{\Q}{\mathbb{Q}}
\newcommand{\E}{\mathcal{E}}
\newcommand{\h}{\mathrm{H}}
\newcommand{\GL}{\mathrm{GL}}
\newcommand{\Oh}{\mathrm{O}}

\begin{document}

\title{On the average of $p$-Selmer rank in quadratic twist families of Elliptic curves over function field}

\author{Sun Woo Park and Niudun Wang}
\address{Department of Mathematics, University of Wisconsin --
  Madison, 480 Lincoln Dr., Madison, WI 53706, USA}
\email{\href{mailto:spark483@math.wisc.edu}{spark483@math.wisc.edu}}
\email{\href{mailto:nwang66@math.wisc.edu}{nwang66@math.wisc.edu}}
\date{\today}

\maketitle

\begin{abstract}
We show that if the quadratic twist family of a given elliptic curve over $\F_q[t]$ with $\mathrm{Char}(\F_q) \geq 5$ has an element whose $\neron$ model has a multiplicative reduction  away from $\infty$, then the average $p$-Selmer rank is $p+1$ in large $q$-limit for almost all primes $p$.
\end{abstract}

\section{Introduction}
Let $\F_q$ be a finite field with $\mathrm{Char}(\F_q) \geq 5$, let $C = \mathbb{P}^1/\F_q$ and $K=\F_q(C)=\F_q(t)$. Say $E: y^2= x^3+A(t)x+B(t)$ is a non-isotrivial elliptic curve defined over $\F_q[t]$. Define the canonical (naive) height of the elliptic curve as follows, where $E'$ is any elliptic curve isomorphic to $E$ of the form $y^2 = x^3 + C(t)x + D(t)$.
\begin{equation*}
    h(E) := \text{inf}_{E' \cong E} \left( \text{max} \{ 3 \deg C, 2 \deg D \} \right)
\end{equation*}
Let $E_f$ be the quadratic twist of $E$ by square-free polynomial $f(t) \in \F_{q}[t]$. 
\begin{equation*}
    E_f: f(t)y^2=x^3+A(t)X+B(t)
\end{equation*}
Let $M(n,q)$ be the set of square-free polynomials over $\F_{q}$ such that $h(E_f) \leq n$.  

Poonen-Rains heuristic shows that the average $p$-Selmer rank of elliptic curves over a global field $k$ is $p+1$(see \cite{Poonen12}.). It is natural to ask whether the same heuristic is reasonable for the family of quadratic twists of a fixed elliptic curve. We denote by $\mathbb{E}_{n,p}$ the average $p$-Selmer rank over those in the family of quadratic twists with canonical height at most $n$, namely:
\begin{equation*}
    \mathbb{E}_{n,p} = \frac{\sum_{f \in M(n,q)}|\Sel E_f|}{|M(n,q)|}
\end{equation*}
In this paper, we show that under certain assumption on the quadratic twist family of the fixed elliptic curve $E$, the average size of p-Selmer groups is  $p+1$ in large $q$ limit. In particular, we can assume for some large enough $q$ such that the discriminant $\Delta_E$ of $E$ splits in $\F_q[t]$, there exists a quadratic twist $E_0$ of $E$ with minimal height among the quadratic twist family.
\begin{theorem}
Let $E$ be an elliptic curve defined over $K=\F_q(C)=\F_q(t)$ such that there exists at least one quadratic twist of $E$ whose \neron model admits a multiplicative reduction away from $\infty$. Let $E_0$ be the quadratic twist of $E$ with minimal height among the family of quadratic twists of $E$. Then for all primes $p \geq 15$, and coprime to $q$ and all local Tamagawa factors of $E_0$, we have the following equation.
\begin{equation*}
    \lim_{n \to \infty} \lim_{q \to \infty} \mathbb{E}_{n,p} = \lim_{n \to \infty} \lim_{q \to \infty} \frac{\sum_{f \in M(n,q)}|\Sel E_f|}{|M(n,q)|} = p+1
\end{equation*}
\end{theorem}

\begin{remark}
The main theorem shows that for all but finitely many primes $p$, the average $p$-Selmer rank in the large $q$-limit over the family of quadratic twists of $E$ is $p+1$.
\end{remark}

\begin{remark}
While writing the paper, we learned the contemporaneous results from Aaron Landesman on a similar problem. Given a universal family of elliptic curves over $\F_q[t]$ with $q$ coprime to $6n$, the geometric average size of $n$-Selmer group of the universal family is equal to sum of divisors of $n$ as $q \to \infty$. We refer to \cite{Aaron18} for more details. 
\end{remark}

In subsequent sections, we will calculate $\lim_{n \to \infty} \lim_{q \to \infty} \mathbb{E}_{n,k,p}$ as follows. Fix an elliptic curve $E$ over $\F_q[t]$. Let $F_{d,E}$ be the set of square-free polynomials $f$ of degree $d$ over $\bar{\F}_q$ such that $f$ is coprime to $\Delta_E$, the discriminant of $E$. Chris Hall's construction of \'{e}tale $\F_l$-lisse sheaf over $F_{d,E}$ gives the average size of $\Sel(E_f)$ for a subfamily of quadratic twists of $E$. We then order the family of quadratic twists of $E$ by the canonical height $h(E_f)$ which enables us to calculate the average size of $\Sel(E_f)$ in large $q$-limit.

\subsection*{Acknowledgements}
The authors would like to sincerely appreciate Jordan Ellenberg for introducing the problem, patiently explaining various mathematical backgrounds, and giving constructive comments and suggestions on possible ways to approach the problem. The authors would like to thank Chris Hall for explaining via email about the construction of \'{e}tale $\F_p$-lisse sheaf in his paper \cite{CHall}. The authors would also like to thank Aaron Landesman and Soumya Sankar for helpful and insightful discussions.

\section{Monodromy Group}

In this section, we briefly discuss the main machinery used to prove Theorem 1.1. This section follows closely to chapter 2 and 3 of \cite{JEllenberg}. Throughout this paper, we denote by $X_{\bar{\F}_q}$ the base change $X \times_{\F_q} \bar{\F}_q$ where $X$ is a scheme over $\F_q$.

We start with a brief exposition on the moduli space of a family of quadratic twists of $E$ by polynomials $g \in \F_q[t]$ of degree $n$ such that $(g,\Delta_E)=1$. A polynomial $g(x) = a_0x^n+a_1x^{n-1}+\cdots+a_n$ of degree $n$ corresponds to a point in the affine space $\A^{n+1}$ with coordinates $(a_0,a_1,a_2,\dots,a_n)$. Note that the square-free polynomials are parameterized by the set of points on $\A^{n+1}$ where $\mathrm{Disc}(g)$ does not vanish, while $(g,\Delta_E)=1$ amounts to $(a_0,a_1,a_2,\dots,a_n)$ not on the zero locus given by the resultant of $g$ and $\Delta_E$. Thus those square free polynomials $g \in \F_q[t]$ with $(g,\Delta_E)=1$ are parameterized by an open subscheme of $\A^{n+1}$, denoted by $F_n$. It is reasonable to expect that it suffices to compute the average $p$-Selmer rank on the elliptic curves parameterized by the open subscheme $F_n$. We will explain in later sections how we can bound the average $p$-Selmer rank on those quadratic twists parametrized by the complement of $F_n$.

Suppose there exists an \'{e}tale cover $X \rightarrow F_n$ such that the number of $\F_q$-points on the geometric fiber of $X$ at $f \in F_n(\mathbb{F}_{q})$ equals to the size of $\Sel E_f$. Then we have the following equation. 
\begin{equation*}
|X(\F_q)|= \sum_{f\in F_n(\F_q)} |\Sel E_f|
\end{equation*}

On the other hand, the Grothendieck-Lefschetz trace formula gives an explicit equation of the number of $\F_q$-points on $X_{\bar{\F_q}}$(\cite{milneLEC}.)

\begin{equation*}
|X(\F_q)|= \sum_i (-1)^i \Tr \; \Frob \; | \; \h^i_{\et;c}(X_{\bar{\F_q}},\Q_l)|
\end{equation*}

$F_n$ is an open subscheme of $\A^{n+1}$ implies that $|F_n(\F_q)|= q^{n+1}-\Oh(q^n)$. The leading term of $F_n(\F_q)$ is $q^{n+1}$. Hence, computing the average size of $\Sel$ in large $q$-limit amounts to computing the following equation.

\begin{equation*}
\lim_{q \to \infty} \frac{\sum_{f \in F_n(\F_q)}|\Sel E_f|}{|F_n(\F_q)|} =  \lim_{q \to \infty} q^{-(n+1)} |X(\F_q)|
\end{equation*}
The Weil bounds imply that the eigenvalues of the Frobenius action $\Frob$ on $ \h^i_{\et ;c}(X_{\bar{\F}_q};\Q_l)$ have absolute value  bounded by $q^{n+1-\frac{i}{2}}$. Thus for a fixed $n$, if $q$ becomes sufficiently large, any cohomology term other than  $ \h^{0}_{\et ;c}(X_{\bar{\F}_q};\Q_l)$ vanishes. Note that $ \h^{0}_{\et ;c}(X_{\bar{\F}_q};\Q_l)$ is the $\Q_l$ vector space spanned by the irreducible components of $X_{\bar{\F}_q}$. Hence, the following observation holds.
 
\begin{equation}
\lim_{q \to \infty} q^{-(n+1)} |X(\F_q)|= \textrm{\# of geometric irreducible components of $X$ rational over $\F_q$}
\end{equation}

Let $f \in F_n$ be a fixed basepoint.  we have the following short exact sequence.
\begin{equation*}
1 \longrightarrow \pi_1((F_n)_{\bar{\F_q}}, f) \longrightarrow \pi_1(F_n, f) \longrightarrow \mathrm{Gal}(\bar{\F_q}/\F_q) \longrightarrow 1
\end{equation*}

For any $f\in F_n$, the geometric fiber of $X$ at $f$ is an $\F_p$-vector space $X_f$. Observe that $\pi_1(F_n, f)$ acts linearly on $X_f$. Hence we can define the monodromy group of the cover $X \to F_n$ as the image of $\pi_1(F_n, f)$ in $\GL(X_{f})$ and the geometric monodromy group as the image of $\pi_1((F_n)_{\bar{\F_q}},f)$. Denote by $\Gamma$ the monodromy group, and denote by $\Gamma_0$ the geometric monodromy group. This gives us another short exact sequence, where $[q]$ corresponds to the class in $\Gamma/\Gamma_0$ corresponding to the image of the Frobenius $\Frob \in \Gal(\bar{\F}_q / \F_q)$, and $[q]^\mathbb{Z}$ corresponds to a subgroup of $\Gamma/\Gamma_0$ generated by $[q]$.

\begin{equation*}
1 \longrightarrow \Gamma _0 \longrightarrow \Gamma \longrightarrow [q]^\Z \longrightarrow 1
\end{equation*}

We state and prove the following observations, which gives a geometric interpretation of equation (2.1).
\begin{lemma}
The geometric irreducible components of $X$ are in bijection with the orbits of the geometric monodromy group on $X_{f}$.
\end{lemma}
\begin{proof} 
Note that $X_{\bar{\F}_q}$ is \'{e}tale over $(F_n)_{\bar{\F}_q}$ because $X$ is \'{e}tale over $F_n$. Hence, $\pi_1((F_n)_{\bar{\F}_q})$ acts on $X_{\bar{\F_q}}$. The group action preserves each irreducible component of $X_{\bar{\F}_q}$ since each component is \'{e}tale over $(F_n)_{\bar{\F}_q}$, hence preserved by the functoriality of $\pi_1$. Therefore, under the action of $\pi_1((F_n)_{\bar{\F}_q})$, each orbit of the geometric monodromy group on $X_{f}$ would lie inside one irreducible component of $X_{\bar{\F}_q}$. On the other hand, $\pi_1(X_{\bar{\F}_q})$ acts transitively on the geometric fibers of $f$ within an irreducible geometric component, which yields the bijection.
\end{proof}

\begin{lemma}
The action of the Frobenius on the geometric components is given by the action of $[q]$ on the orbits of $\Gamma_0$.
\end{lemma}

This comes directly from the bijection between the geometric irreducible components of $X$ and the orbits of $\Gamma_0$. Therefore, in order to compute the number of geometrically irreducible components of $X$, it suffices to understand the geometric monodromy group $\Gamma_0$ and compute the number of $\Gamma_0$-orbits on $X_f$ which are fixed by $[q]$. Therefore, equation (2.1) can be rewritten as follows.

\begin{equation}
    \lim_{q \to \infty} q^{-(n+1)} |X(\F_q)|= \textrm{\# of orbits of $\Gamma_0$ fixed by $[q]$}
\end{equation}

\section{Construction of Moduli Space} 

\subsection{Cohomology Groups of \neron Models}
In this subsection, we prove several claims which will help us with constructing the desired moduli space discussed in remark 2.3. 

Let $q$ be a power of prime $q = q_0^k$ such that $q_0$ is not divisible by 2 and 3. Fix an algebraic closure $\F_q \rightarrow \bar{\F_q}$. Let $C/\F_q = \mathbb{P}^1/\F_q$, and let $K=\F_q(C) = \F_q(t)$ be its function field. Fix an elliptic curve $E$ over $K$, and let $E_f$ be the quadratic twist of the elliptic curve by $f \in \Conf(\F_q)$. Let $\E \rightarrow C$ be the \neron model for the elliptic curve $E$. For any prime $p$ that is invertible in $K$, the multiplication by $p$ map on $E_f(K)$ extends uniquely to an isogeny  $\times p: \E \rightarrow \E$. Define $\E_{p}$ to be the kernel of $\times p$.

We state a result from \cite{KCesnavicius}, which states that the first cohomology group of $\mathcal{E}_{p}$ and $\Sel$ are isomorphic under certain arithmetic conditions. First, we state the following lemma from Appendix B of \cite{KCesnavicius}.

\begin{lemma}
Let $S$ be a connected Noetherian normal scheme of dimension $\leq 1$. Let $\bar{K}$ be the function field of $S$, and for every point $s \in S$, let $k(s)$ be the residue field of $s$. Let $A \to \mathrm{Spec} \bar{K}$ and $B \to \mathrm{Spec} \bar{K}$ be abelian varieties, and let $\mathcal{A} \to S$ and $\mathcal{B} \to S$ be their \neron models. Suppose $\phi: A \to B$ is a $\bar{K}$-isogeny of abelian varieties. Denote by $\tilde{\phi}: \mathcal{A} \to \mathcal{B}$ the map induced on \neron models over $S$. If $\mathcal{A}$ has semiabelian reduction at all the nongeneric $s \in S$ with $\mathrm{Char}(k(s)) | \deg \phi$, then $\tilde{\phi}: \mathcal{A} \to \mathcal{B}$ is flat.
\end{lemma}
\begin{proof}
See \cite{KCesnavicius} lemma B.4.
\end{proof}

In particular, the lemma above shows that the map $\times p: \E \to \E$ is flat if $q$ is coprime to $p$. Using this implication, we can state the following application of Proposition 5.4 from \cite{KCesnavicius}.

\begin{theorem}
Let $\phi: A \to B$ be a morphism of abelian varieties. Let $\mathcal{A}$ and $\mathcal{B}$ be their \neron models. Denote by $A[\phi]$ the kernel of $\phi: A \to B$, and denote by $\mathcal{A}[\phi]$ the kernel of $\phi: \mathcal{A} \to \mathcal{B}$. Let $\mathrm{Sel}_\phi A$ be the $\phi$-Selmer group of $A$. Suppose the morphism $\phi: \mathcal{A} \to \mathcal{B}$ is flat. If $\deg \phi$ is coprime to any local Tamagawa factors of $A$ and $B$, and $2$ does not divide $\deg \phi$, then the following equation holds inside $\h^1_{\fppf}(K, A[\phi])$.
\begin{equation*}
    \h^1_{\fppf}(S, \mathcal{A}[\phi]) = \mathrm{Sel}_\phi A
\end{equation*}
\end{theorem}
\begin{proof}
See \cite{KCesnavicius} proposition 5.4.
\end{proof}

More specifically, the theorem above implies that when $p$ is coprime to $2$, $q$, and the local Tamagawa factors of $E$, then there exists an isomorphism between $\h^1_{\fppf}(C, \E_p)$ and $\Sel E$.

In fact, under certain conditions on $E$, we can extend the results of Theorem 3.2 to the family of quadratic twists of $E$. Let $E_f$ be the quadratic twist of the elliptic curve by square-free polynomials $f \in \F_q[t]$. Denote by $\E_f \rightarrow C$ the \neron model for the elliptic curve $E_f$. Let $\E_{f,p}$ be the kernel of multiplication by $p$ map over $\E_f$.

\begin{corollary}
Let $E/K: y^2 = x^3 + Ax + B$ be an elliptic curve, and let $\Delta_E$ be the discriminant of $E$. Suppose that no prime factors $\pi$ of $\Delta_E$ satisfy the condition that $\pi^2 | A$ and $\pi^3 | B$. Let $p$ be a prime such that $p$ is coprime to $2,3,q_0$ and all the local Tamagawa factors of $E$. Then the following isomorphism holds for any square-free polynomial $f \in \F_q[t]$.
\begin{equation*}
\h^1_{\fppf}(C,\E_{f,p})=\Sel E_f
\end{equation*}
\end{corollary}
\begin{proof}
For an explicit calculation of local Tamagawa factors using Tate's algorithm, see \cite{Silverman91} Chapter 4 Section 9. Say $f=(\pi_1\dots \pi_s)g$, where $J=\{\pi_1\dots \pi_s\}$ are prime factors of $\Delta_E$ and $(g, \Delta_E)=1$, We first note that the conditions on $E$ imply that the discriminant of the twist $E_f$ is equal to $f^6 \Delta_E$. Tate's algorithm shows $E_f$ has additive reduction on all primes $\pi$ dividing $g$. Therefore, all local Tamagawa factors arising from such $\pi$'s are at most $4$. 

On the other hand, the additive reductions $\pi_i \in J$ of $E$ will stay as additive reductions of $E_f$, while the multiplicative reductions $\pi_j \in J$ will all become additive reductions of $E_f$.  For all the other primes $\rho | \Delta_E$ but not in $J$, $v_\rho(\Delta_E)$ and $v_\rho(\Delta_{E_f})$ are the same. Therefore, for any $\rho | \Delta_E$, no matter weather $\rho$ divides f or not, we will have the local Tamagawa factor of $E_f$ at $\rho$ either equals to the local
Tamagawa factor of $E$ or equals to 1,2,3. Then we can apply theorem 3.2. to $E_f$.

\end{proof}

\begin{remark}
For such an elliptic curve $E$ in the above Corollary, we can actually conclude that if the \neron model of $E$ itself has no multiplicative reduction away from $\infty$, there is no quadratic twists of $E$ whose \neron model admits a multiplicative reduction away from $\infty$. This follows exactly from the fact that additive reductions $\pi_i \in J$ of $E$ will stay as additive reductions of $E_f$.

\end{remark}

The theorem hence implies that for all but finitely many primes $p$, the following isomorphism holds.
\begin{equation*}
\h^1_{\fppf}(C,\E_{f,p})=\Sel E_f
\end{equation*}

Since $\E_{f,p}$ is a smooth commutative group scheme, we know that $\h^1_{\fppf}(C,\E_{f,p})$ is isomorphic to $\h^1_{\et}(C,\E_{f,p})$. (See \cite{poonenRPV} Remark 6.6.3) 

We now examine whether there is a way to explicitly compute the size of $\h^1_{\et}(C, \E_{f,p})$. Chris Hall gives an explicit computation of the \'{e}tale cohomology groups of $\E_{f,p}$ over $C_{\bar{\F}_q}$ under certain conditions on the size of the Galois group $\Gal(K(E[p])/K)$.

\begin{definition}
Let $E/K$ be an elliptic curve. The geometric Galois group $H_p$ is the subgroup of the Galois group $\Gal(K(E[p])/K)$ whose fixed field is $(K(E[p]) \cap \bar{\F}_q)/K$ given by adjoining a primitive $p$th root of unity. We say that $E$ has big monodromy at $p$ if the geometric Galois group contains $\mathrm{SL}_2(\F_p)$.
\end{definition}

In fact, for any prime $p \geq 5$, any twist $E_f$ has big monodromy at $p$ if and only if $E$ has big monodromy at $p$. This is because $\mathrm{SL}_2(\F_p)$ does not have index $2$ subgroups, so $K(E_f[p])$ and $k(\sqrt{f})$ are geometrically disjoint extensions of $K$.

\begin{theorem}
Let $C/\F_q$ be a proper smooth geometrically connected curve, and let $K$ be its function field. Let $E/K$ be a non-isotrivial elliptic curve. Then there exists a constant $c(K)$ such that $E$ has big monodromy at $p$ for any $p \geq c(K)$ and $p$ coprime to $\mathrm{Char}(K)$. The constant $c(K)$ is defined as follows.
\begin{equation*}
    c(K) := 2 + \max \{l \; | \; l \; \mathrm{ is} \; \mathrm{prime }, \frac{1}{12} \left(l - (6 + 3e_2 + 4e_3) \right) \leq \mathrm{genus}(C)\}
\end{equation*}
Here, $e_2$ and $e_3$ are constants defined as follows.
\begin{align*}
    e_2 &= \begin{cases}
    1 & \mathrm{if} \; l \equiv 1 \; \mathrm{mod} \; 4 \\
    -1 & \mathrm{otherwise} 
    \end{cases} \\
    e_3 &= \begin{cases}
    1 & \mathrm{if} \; l \equiv 1 \; \mathrm{mod} \; 3 \\
    -1 & \mathrm{otherwise}
    \end{cases}
\end{align*}
\end{theorem}
\begin{proof}
See \cite{CHallSel} theorem 1.1.
\end{proof}

In particular, let $C = \mathbb{P}^1$ and $K$ be the function field of $C$. Fix a non-isotrivial elliptic curve $E/K$. Let $\{E_f\}$ be the family of quadratic twists of $E$, where $f$ is any square-free polynomial over $\F_q$. The theorem above implies that $E/K$ has big monodromy at $p$ for any prime $p \geq 15$ and coprime to $q$. Hence, for any prime $p \geq 15$ and coprime to $q$, the twist $E_f/K$ also has big monodromy at $p$. Under the aforementioned conditions on $p$, we can give an explicit calculation of the \'{e}tale cohomology group $\h^1(C_{\F_q}, \E_{f,p})$ for any $E_f$.

\begin{lemma}
Let $C/\F_q$ be a proper smooth geometrically connected curve, and let $K$ be its function field. Fix an elliptic curve $E/K$. Let $p$ be a prime such that $E$ has big monodromy at $p$. Then for any square-free polynomial $f$ over $\F_q$, the \'{e}tale cohomology groups of $\mathcal{E}_{f,p}$ over $C_{\bar{\mathbb{F}}_q}$ are $\mathbb{F}_p$-vector spaces with the following dimensions.
\begin{equation*}
    \dim_{\mathbb{F}_p} (\h^i_{\et}(C_{\bar{\F}_q}, \mathcal{E}_{f,p})) = 
    \begin{cases}
    \deg (M_f) + 2 \deg (A_f) - 4 (\mathrm{genus}(C) - 1) & \mathrm{ if } \; i = 1 \\
    0 & \mathrm{otherwise}
    \end{cases}
\end{equation*}
Here, $M_f$ and $A_f$ are the divisors of multiplicative and additive reduction of $\mathcal{E}_{f,p} \to C$.
\end{lemma}
\begin{proof}
See \cite{CHall} lemma 5.2.
\end{proof}

In particular, if $C = \mathbb{P}^1$, then the dimension of $\h^1(C_{\F_q}, \E_{f,p})$ as an $\F_p$-vector space is $\deg(M_f) + 2 \deg(A_f) + 4$ for any twist $E_f$.

\begin{remark}
The Weil pairing on $E_f[p]$ induces a non-degenerate skew-symmetric pairing on $\mathcal{E}_{f,p}$. Hence the Weil pairing induces a non-degenerate symmetric pairing on $\h^1_{\et}(C_{\bar{\mathbb{F}}_q}, \mathcal{E}_{f,p})$ as follows.
\begin{align*}
    \h^1_{\et}(C_{\bar{\mathbb{F}}_q}, \mathcal{E}_{f,p}) \times \h^1_{\et}(C_{\bar{\mathbb{F}}_q}, \mathcal{E}_{f,p}) & \to H^2_{\et}(C_{\bar{\mathbb{F}}_q}, \mathcal{E}_{f,p} \otimes \mathcal{E}_{f,p}) \\
    &\to H^2_{\et}(C_{\bar{\mathbb{F}}_q}, \mathbb{F}_p(1))
\end{align*}
The first map comes from the cup product of cohomology classes and Poincar\'{e} duality, while the second map comes from the induced Weil pairing $\mathcal{E}_{f,p} \times \mathcal{E}_{f,p} \to \mathbb{F}_p(1)$. Note that the following isomorphism holds. (See \cite{milneLEC} Chapter 14.)
\begin{equation*}
    H^2_{\et}(C_{\bar{\mathbb{F}}_q}, \mathbb{F}_p(1)) \cong \mathbb{F}_p
\end{equation*}
Therefore, the Weil pairing on $E_f[p]$ induces a non-degenerate symmetric bilinear pairing of first \'{e}tale cohomology groups.
\begin{equation*}
    \h^1_{\et}(C_{\bar{\mathbb{F}}_q}, \mathcal{E}_{f,p}) \times \h^1_{\et}(C_{\bar{\mathbb{F}}_q}, \mathcal{E}_{f,p}) \to \mathbb{F}_p
\end{equation*}
\end{remark}

Using the above lemma and the Leray spectral sequence, we can derive a relation between \'{e}tale cohomology groups of $\E_{f,p}$ over $C$ and those over $C_{\bar{\F}_q}$.

\begin{theorem}[Leray Spectral Sequence]
Let $\phi: Y \to X$ be a morphism of schemes, and let $\mathcal{F}$ be a sheaf on $Y_{et}$. Then there exists the following spectral sequence.
\begin{equation*}
E^2_{r,s} = \h^r_{\et} (X, R^s \phi_* \mathcal{F}) \Rightarrow \h^{r+s}_{\et} (Y, \mathcal{F})
\end{equation*}
\end{theorem}
\begin{proof}
See \cite{milneLEC} theorem 12.7.
\end{proof}

\begin{lemma}
Fix an elliptic curve $E/K: y^2 = x^3 + Ax + B$ such that no prime factors $\pi$ of $\Delta_E$ satisfy the condition that $\pi^2 | A$ and $\pi^3 | B$. Suppose $p$ is a prime satisfying these conditions.
\begin{enumerate}
    \item $p \geq 15$ 
    \item $(p,\mathrm{Char}(K)) = 1$ 
    \item $p$ does not divide any local Tamagawa factors of $E$.
\end{enumerate}
Then for any quadratic twist $E_f$, the following isomorphism exists.
\begin{equation*}
    H^1_{\et}(C_{\bar{\F}_q}, \mathcal{E}_{f,p})^{\Gal(\bar{\mathbb{F}}_q / \mathbb{F}_q)} \cong \Sel E_f
\end{equation*}
\end{lemma}
\begin{proof}
Consider the morphism of schemes $\phi: C \to \Spec (\mathbb{F}_q)$. By the Leray spectral sequence, the following spectral sequence exists.
\begin{equation*}
    E^2_{r,s} = \h^r (\mathbb{F}_q, \h^s_{\et}(C_{\bar{\F}_q}, \mathcal{E}_{f,p})) \Rightarrow \h^{r+s}_{\et} (C, \mathcal{E}_{f,p})
\end{equation*}
By lemma 3.7, the cohomology group $\h^r (\mathbb{F}_q, \h^s_{\et}(C_{\bar{\F}_q}, \mathcal{E}_{f,p}))$ is trivial whenever $s \neq 1$. Hence, the entries of the $E^2$ page of the spectral sequence are given as follows.

\begin{equation*}
\begin{tikzpicture}
\matrix (m) [matrix of math nodes,
             nodes in empty cells,
             nodes={minimum width=5ex, minimum height=6ex,
                    text depth=1ex,
                    inner sep=0pt, outer sep=0pt,
                    anchor=base},
             column sep=2ex, row sep=2ex]%
{
r & 0 & \h^r (\mathbb{F}_q, \h^1_{\et}(C_{\bar{\F}_q}, \mathcal{E}_{f,p})) & 0 & \cdots & 0  \\
    [5ex,between origins]
\vdots  & \vdots  & \vdots       & \vdots       & \ddots & \vdots  \\
    [3ex,between origins]
    2   &   0     & \h^2 (\mathbb{F}_q, \h^1_{\et}(C_{\bar{\F}_q}, \mathcal{E}_{f,p})) & 0 & \cdots & 0  \\
   1   &   0     & \h^1 (\mathbb{F}_q, \h^1_{\et}(C_{\bar{\F}_q}, \mathcal{E}_{f,p})) & 0 & \cdots & 0 \\
    [5ex,between origins]
   0   &   0     &  \h^0 (\mathbb{F}_q, \h^1_{\et}(C_{\bar{\F}_q}, \mathcal{E}_{f,p}))           &  0           &   \cdots     &  0    \\
    [3ex,between origins]
        &  0     &  1           &  2           & \cdots &  s           & \strut \\
};
\draw[thick] (m-1-1.north east) -- (m-6-1.east) ;
\draw[thick] (m-6-1.north) -- ($(m-5-6.east)!0.5!(m-6-6.east)$) ;
\end{tikzpicture}
\end{equation*}

The Leray spectral sequence implies the following isomorphism. 
\begin{equation*}
    \h^0 (\mathbb{F}_q, \h^1_{\et}(C_{\bar{\F}_q}, \mathcal{E}_{f,p})) \cong \h^1_{\et}(C, \mathcal{E}_{f,p})
\end{equation*}
Recall the following isomorphism.
\begin{equation*}
    \h^1_{\et}(C, \mathcal{E}_{f,p}) \cong \h^1_{\fppf}(C, \mathcal{E}_{f,p})
\end{equation*}
Note that the $0$-th cohomology group is precisely the fixed subgroup of $\h^1_{\et}(C_{\bar{\F}_q}, \mathcal{E}_{f,p})$ by $\Gal(\bar{\mathbb{F}}_q / \mathbb{F}_q)$.
\begin{equation*}
    \h^1_{\et}(C_{\bar{\mathbb{F}}_q}, \mathcal{E}_{f,p})^{\Gal(\bar{\mathbb{F}}_q / \mathbb{F}_q)} \cong \h^0 (\mathbb{F}_q, \h^1_{\et}(C_{\bar{\mathbb{F}}_q}, \mathcal{E}_{f,p}))
\end{equation*}
Corollary 3.3 implies the following isomorphism holds for all but finitely many $p$.
\begin{equation*}
    \h^1_{\fppf} (C, \mathcal{E}_{f,p}) \cong \Sel E_f
\end{equation*}
Hence, for all but finitely many $p$, the following isomorphism holds.
\begin{equation*}
    \h^1_{\et}(C_{\bar{\F}_q}, \mathcal{E}_{f,p})^{\Gal(\bar{\mathbb{F}}_q / \mathbb{F}_q)} \cong \Sel E_f
\end{equation*}
\end{proof}

\subsection{Construction of \'{E}tale $\F_p$-lisse Sheaf}

In this subsection, we follow through Chris Hall's construction of \'{e}tale $\F_p$-lisse sheaf over a subset of square-free polynomials $f$ of fixed degree over $\F_q$. The construction will help us calculate the average size of $\Sel(E_f)$ for a subfamily of quadratic twists of $E$.

As before, let $C = \mathbb{P}^1$ over $\F_q$, and let $K$ be the function field of $C$. Fix a non-isotrivial elliptic curve $E/K: y^2 = x^3 + Ax + B$. We recall the construction of the space $F_n$ over $\bar{\F}_q$, as mentioned in section 2.
\begin{equation*}
    F_n = \{ f \in \bar{\F}_{q}[t] \; | \; f \; \textrm{is square-free}, \deg f = n, (f, \Delta_E) = 1 \}
\end{equation*}

As constructed in \cite{CHall} (See Chapter 5.3), consider the \'{e}tale $\mathbb{F}_p$-lisse sheaf $\tau_{n,p,E} \to F_n$ whose geometric fiber over $f \in F_n(\mathbb{F}_{q})$ is $\h^1(\Conf_{\bar{\F}_q}, \mathcal{E}_{f,p})$. Note that Chris Hall's construction of $\tau_{n,p,E}$ is an $\F_p$-analogue of Katz's construction of \'{e}tale $\bar{\mathbb{Q}}_p$-lisse sheaves using middle convolutions. We refer to  \cite{Katz98} proposition 5.2.1. for a detailed explanation on the construction of the \'{e}tale $\bar{\mathbb{Q}}_p$-lisse sheaves.

We state the following theorem by Chris Hall, which gives an explicit computation of the geometric monodromy group of $\tau_{n,p,E} \to F_n$.

\begin{theorem}
Let $E$ be an elliptic curve over $K$ such that there exists at least one quadratic twist of $E$ whose \neron model admits a multiplicative reduction away from $\infty$. Let $p$ be a prime such that $E$ has big monodromy at $p$, i.e. $p \geq 15$. Let $O(\h^1_{\et}(\Conf_{\bar{\F}_q}, \mathcal{E}_{f,p}))$ be the orthogonal group of $\h^1_{\et}(\Conf_{\bar{\F}_q}, \mathcal{E}_{f,p})$ which preserves the non-degenerate symmetric bilinear pairing $\mu$. (See Remark 3.7 for the construction of $\mu$.) 

Then the geometric monodromy group of $\tau_{n,p,E} \to F_n$ is isomorphic to a subgroup of the orthogonal group $O(\h^1_{\et}(\Conf_{\bar{\F}_q}, \mathcal{E}_{f,p}))$ of index at most $2$ and is not isomorphic to $SO(\h^1_{\et}(\Conf_{\bar{\F}_q}, \mathcal{E}_{f,p}))$.
\end{theorem}
\begin{proof}
See \cite{CHall} Theorem 5.3
\end{proof}

Note that the commutator subgroup of $O(\h^1_{\et}(\Conf_{\bar{\F}_q}, \mathcal{E}_{f,p}))$ is of index $4$. Hence, in order to understand the number of orbits of the desired geometric monodromy group, it suffices to understand the number of orbits of both $O(\h^1_{\et}(\Conf_{\bar{\F}_q}, \mathcal{E}_{f,p}))$ and its commutator subgroup. We finish this section with the following lemma, which describes the number of the orbits of the aforementioned two groups.

\begin{lemma}
Let $V$ be an $\mathbb{F}_p$-vector space of dimension $d$ where $\mathrm{char}(\mathbb{F}_p) \neq 2$. Given a non-degenerate symmetric bilinear pairing $\mu: V \times V \to \mathbb{F}_p$, let $O(V)$ be the orthogonal group of $V$. Then the number of orbits of $O(V)$ is $p+1$, and the number of orbits of the commutator subgroup $[O(V), O(V)]$ is $p+1$ if $d \geq 5$.
\end{lemma}
\begin{proof}
We first show that the number of orbits of $O(V)$ on $V$ is $p+1$ for any $d$. It suffices to show that the orthogonal group acts transitively on the set of nonzero vectors of a given norm. Suppose $v,w \in V$ are two non-zero vectors of the same norm. Then there exists an isometry $\phi: Span(v) \to Span(w)$ given by $v \mapsto w$. By Witt's Theorem, $\phi$ extends to an isometry $\tilde{\phi}:V \to V$, which proves the claim. Note that the orthogonal group has $1$ orbit on the set of vectors of non-zero norm, and $2$ orbits on the set of vectors norm zero. In the latter case, the two orbits are $\{0\}$ and the set of non-zero vectors of norm zero. 

We now show that the number of orbits of the commutator subgroup $[O(V), O(V)]$ on $V$ is $p+1$ for $d \geq 5$. Again, it suffices to show that the commutator subgroup $[O(V), O(V)]$ acts transitively on the set of nonzero vectors of a given norm. Suppose $v,w \in V$ are two non-zero vectors of the same norm. Then by the aforementioned argument, there exists an isomtery $\tilde{\phi} \in O(V)$ such that $\tilde{\phi}(v) = w$. We want to show that there exists $\tilde{\psi}, \tilde{\varphi} \in O(V)$ such that the following equation holds.
\begin{equation*}
    \tilde{\phi} (v) = \tilde{\psi} \tilde{\varphi} \tilde{\psi}^{-1} \tilde{\varphi}^{-1} (v)
\end{equation*}
It suffices to consider the case when $\mu(v,w) = 0$. If not, then $v = w$ and we can take $\tilde{\phi}$ to be the identity element in $O(V)$. Suppose $v,w$ are orthogonal. Let $\{v,w,u_1, u_2, \cdots, u_{d-2}\}$ be the orthogonal basis of $V$. Without loss of generality, we can assume that the norms of the basis vectors are the same. 

Suppose $d \geq 5$. Let $W$ be the span of $\{v, u_1, u_2, u_3, w\}$. Then consider the following isometries $\psi, \varphi:W \to W$. Here, the matrices are given with respect to the orthogonal basis $\{v, u_1, u_2, u_3, w\}$.
\begin{equation*}
    \psi = \begin{pmatrix}
    0 & 1 & 0 & 0 & 0 \\
    1 & 0 & 0 & 0 & 0 \\
    0 & 0 & 0 & 1 & 0 \\
    0 & 0 & 1 & 0 & 0 \\
    0 & 0 & 0 & 0 & 1
    \end{pmatrix}, \; \; \; \; \; \; \;
    \varphi = \begin{pmatrix}
    1 & 0 & 0 & 0 & 0 \\
    0 & 0 & 1 & 0 & 0 \\
    0 & 1 & 0 & 0 & 0 \\
    0 & 0 & 0 & 0 & 1 \\
    0 & 0 & 0 & 1 & 0
    \end{pmatrix}
\end{equation*}
Note that $\psi^{-1} = \psi$ and $\varphi^{-1} = \varphi$ as isometries over $W$. The matrix form of $\psi \varphi \psi^{-1} \varphi^{-1}$ is given as follows, which implies that $\psi \varphi \psi^{-1} \varphi^{-1}$ maps $v$ to $w$.
\begin{equation*}
    \psi \varphi \psi^{-1} \varphi^{-1} = \begin{pmatrix}
    0 & 0 & 0 & 0 & 1 \\
    0 & 0 & 1 & 0 & 0 \\
    0 & 0 & 0 & 1 & 0 \\
    1 & 0 & 0 & 0 & 0 \\
    0 & 1 & 0 & 0 & 0
    \end{pmatrix}
\end{equation*}
By Witt's theorem, there exist isometries $\tilde{\psi}, \tilde{\varphi} \in O(V)$ such that $\tilde{\psi} \tilde{\varphi} \tilde{\psi}^{-1} \tilde{\varphi}^{-1}$ maps $v$ to $w$. A similar argument as for the case of $O(V)$ shows that the number of orbits of $[O(V),O(V)]$ on $V$ is $p+1$.
\end{proof}

Using the above lemma, we can calculate the average size of $\Sel E_f$ for $f \in F_n(\F_{q})$.

\begin{theorem}
Fix a non-isotrivial elliptic curve $E: y^2 = x^3 + Ax + B$ over $\F_q[t]$ such that there exists at least one quadratic twist of $E$ whose \neron model admits a multiplicative reduction away from $\infty$. Let $E_0$ be the quadratic twist of $E$ with minimal height among the family of quadratic twists of $E$. Let $n$ be an integer such that $n \geq 5$. Let $p$ be a prime such that $p \geq 15$, and coprime to $q$ and all local Tamagawa factors of $E_0$. Then the average size of $\Sel E_f$ for a subfamily of quadratic twists $\{E_f\}_{f \in F_n(\F_{q})}$ is $p+1$ when $q \to \infty$, i.e.
\begin{equation*}
    \lim_{q \to \infty} \frac{\sum_{f \in F_n(\F_{q})} |\Sel(E_f)|}{|F_n(\F_{q})|} = p+1
\end{equation*}
\end{theorem}
\begin{proof}
Denote by $\{E_f\}$ the family of quadratic twists of $E$. Then note the $E_0$ must have at least one place of multiplicative reduction by the proof of Corollary 3.3. Indeed, $E_0$ satisfies the condition for Corollary 3.3. Note that the quadratic twist families $\{E_f\}$ and $\{(E_0)_g\}$ are equal. Hence, we can apply lemma 3.10 to the subfamily of quadratic twists $\{E_f\}_{f \in F_n(\F_{q})}$.

Since $F_n$ is an open subscheme of $\A^{n+1}$, it holds that $|F_n(\F_{q})| = q^{n+1} + O_{n}(q)$. Then the Grothendieck-Lefschetz trace formula (i.e. Section 2) and lemma 3.10 shows the following equation.
\begin{align*}
    \lim_{q \to \infty} \frac{\sum_{f \in F_n(\F_{q})} |\Sel(E_f)|}{|F_n(\F_{q})|} 
    &= \lim_{q \to \infty} \frac{\sum_{f \in F_n(\F_{q})} |\h^1_{\et}(C_{\bar{\F}_q}, \mathcal{E}_{f,p})^{\Gal(\bar{\mathbb{F}}_q / \F_{q})}|}{|F_n(\mathbb{F}_q)|} \\
    &= \lim_{q \to \infty} \frac{|\tau_{n,p,E}(\F_{q})|}{|F_n(\F_{q})|} \\
    &= \lim_{q \to \infty} \# \textrm{ of orbits of } \Gamma_0 \; \textrm{ fixed  by} \; [q]
\end{align*}
We recall that $\Gamma$ is the image of $\pi_1(F_n)$ in $O(\h^1_{\et}(C_{\bar{\mathbb{F}}_q}, \mathcal{E}_{f,p}))$, and $\Gamma_0$ is the image of $\pi_1((F_n)_{\bar{\F}_q})$ in $O(\h^1_{\et}(C_{\bar{\mathbb{F}}_q}, \mathcal{E}_{f,p}))$. The class $[q]$ is the image of the Frobenius $\mathrm{Frob}_{q} \in \Gal(\bar{\F}_q / \F_{q})$ in $\Gamma/\Gamma_0$.

By theorem 3.11, the geometric monodromy group $\pi_1(({F_n})_{\bar{\mathbb{F}}_q})$ is isomorphic to a subgroup of $O(\h^1_{\et}(C_{\bar{\mathbb{F}}_q}, \mathcal{E}_{f,p}))$ of index at most $2$ and is not $SO(\h^1_{\et}(C_{\bar{\mathbb{F}}_q}, \mathcal{E}_{f,p}))$. Recall from remark 3.8 that the Weil pairing on $E_f[p]$ induces a non-degenerate symmetric bilinear pairing $\mu$ on $\h^1_{\et}(C_{\bar{\mathbb{F}}_q}, \E_{f,p})$.
\begin{equation*}
    \h^1_{\et}(C_{\bar{\mathbb{F}}_q}, \mathcal{E}_{f,p}) \times \h^1_{\et}(C_{\bar{\mathbb{F}}_q}, \mathcal{E}_{f,p}) \to \mathbb{F}_p
\end{equation*}
Hence, the frobenius map $\mathrm{Frob}_{q}$ preserves the pairing on $\h^1_{\et}(C_{\bar{\mathbb{F}}_q}, \mathcal{E}_{f,p})$. We apply lemma 3.12 by setting $V = \h^1_{\et}(C_{\bar{\mathbb{F}}_q}, \mathcal{E}_{f,p})$ and $\mu$ to be the non-degenerate symmetric bilinear pairing induced from the Weil pairing over $E_f[p]$. 

Note that the elliptic curve $E_f$ has additive reduction at all primes $\pi$ dividing $f$. Hence, lemma 3.7 implies that for $n \geq 5$, the dimension of the vector space $V$ is greater than $5$. Hence for $n \geq 5$, the orbits of $\pi_1(({F_n})_{\bar{\mathbb{F}}_q})$ are the sets of non-zero vectors of a fixed norm and the set $\{0\}$. Therefore, $\Frob$ preserves the orbits of $\pi_1(({F_n})_{\bar{\mathbb{F}}_q})$. Hence, the number of orbits of $\pi_1(({F_n})_{\bar{\mathbb{F}}_q})$ fixed by $\mathrm{Frob}_{q}$ is $p+1$ for all $q$.

Hence, we have the following equation, which proves the theorem.
\begin{equation*}
    \lim_{q \to \infty} \frac{\sum_{f \in F_n(\F_{q})} |\Sel(E_f)|}{|F_n(\F_{q})|} = p+1
\end{equation*}
\end{proof}

\section{Main Theorem}

In this section, we prove the main theorem by using theorem 3.13. Fix a non-isotrivial elliptic curve $E: y^2 = x^3 + Ax + B$ over $\F_q[t]$ such that there exists at least one quadratic twist of $E$ whose \neron model admits a multiplicative reduction away from $\infty$. We also assume that $\mathrm{char}(F_q) \geq 5$. Denote by $\Delta_E$  the discriminant of the elliptic curve $E$. We then order the family of quadratic twists $\{E_f\}$ with square-free polynomials $f$ over $\F_{q}$ based on the canonical height of $E_f$. 

The idea of the proof is as follows. Let $E_0$ be the quadratic twist of $E$ with minimal height among the family of quadratic twists of $E$. Then by corollary 3.3, the \neron model of $E_0$ admits a multiplicative reduction away from $\infty$. Note the quadratic twist families $\{E_f\}$ and $\{(E_0)_f\}$ are the same. Hence, we can reorder the family $\{E_f\}$ by $\{(E_0)_f\}$. Such reordering of family of quadratic twists will allow us to ensure that the subfamily of quadratic twists whose $p$-Selmer rank is undetermined does not affect the average size of $p$-Selmer group of $\{E_f\}$ as $q \to \infty$.

\begin{remark}
One may ask whether it is possible to compute the average size of $p$-Selmer groups by ordering the family of quadratic twists $\{E_f\}$ by the degree of the twisting polynomial $f$. The problem with this approach is that Chris Hall's construction of \'{e}tale $\F_p$-lisse sheaf only works for elliptic curves $E$ with at least one multiplicative reduction. Suppose $E$ has at least one place of multiplicative reduction. Then the quadratic twist family $\{E_f\}$ can be decomposed into the following two disjoint sets. 
\begin{equation*}
    \{E_f\} = \{E_f\}_{\{f \; | \; (f, \Delta_E) = 1\}} \sqcup \{E_f\}_{\{f \; | \; (f, \Delta_E) \neq 1\}}
\end{equation*}
The subfamily $\{E_f\}_{\{f \; | \; (f, \Delta_E) = 1\}}$, which dominates the family $\{E_f\}$ when $\deg f \to \infty$, consists of  quadratic twists of $E$ having at least one place of multiplicative reduction, the subfamily on which the average $p$-Selmer rank is known to be $p+1$. 

However, the elliptic curve $E' := E_{\Delta_E}$ has no multiplicative reduction. Note that $\Delta_{E'} = \Delta_E^7$. Contrary to $\{E_f\}$, the family of quadratic twists of $\{E'_f\}$ is given as follows.
\begin{equation*}
    \{E'_f\} = \{E'_f\}_{\{f \; | \; (f, \Delta_{E'}) = (f, \Delta_E) = 1\}} \sqcup \{E_f\}_{\{f \; | \; (f, \Delta_{E'}) = (f, \Delta_E) \neq 1\}}
\end{equation*}
Here, the subfamily $\{E'_f\}_{\{f \; | \; (f, \Delta_E) = 1\}}$, which dominates the family $\{E'_f\}$ when $\deg f \to \infty$, consists of quadratic twists of $E'$ having no multiplicative reductions, the subfamily on which the average size of $p$-Selmer group is unknown.

But notice that $\{E_f\}$ and $\{E'_f\}$ are the same family of quadratic twists. Hence, we cannot determine the average size of $p$-Selmer group by ordering the family of quadratic twists based on the degree of the twisting polynomials.
\end{remark}

Before we present the proof of the main theorem, we state the following definitions and notations.
\begin{definition}
Denote by $F(n)$ the following scheme defined over $\bar{\F}_q$.
\begin{equation*}
    F(n) := \{f \in \bar{\F}_q [t]\; | \; f \in F_d \; \textrm{for} \; d \leq n\} = \bigsqcup_{d=1}^n F_d
\end{equation*}
\end{definition}
As mentioned before $F(n)$ is an open subscheme of $\A^{n+1}$. Hence, the following equation holds

\begin{equation*}
    |F(n)(\F_{q})| = q^{n+1} - \mathrm{O}(q^n)
\end{equation*}

\begin{definition}
Denote by $\tilde{\tau}(n,p,E)$ the sheaf over $F(n)$ obtained by gluing \'{e}tale $\F_p$-lisse sheaves $\{\tau_{d,p,E} \to F_d\}$ for all $d \leq n$.
\end{definition}

The following theorem states an analogue of Theorem 3.12 for the subfamily of quadratic twists $\{E_f\}_{f \in F(n)(\F_{q})}$ such that $E$ satisfies the aforementioned two conditions.

\begin{theorem}
Fix a non-isotrivial elliptic curve $E: y^2 = x^3 + Ax + B$ over $\F_q[t]$ such that there exists at least one quadratic twist of $E$ whose \neron model admits a multiplicative reduction away from $\infty$. Let $E_0$ be the quadratic twist of $E$ with minimal height among the family of quadratic twists of $E$. Let $n$ be an integer such that $n \geq 5$. Let $p$ be a prime such that $p \geq 15$, and coprime to $q$ and all local Tamagawa factors of $E_0$. Then the average size of $\Sel E_f$ for the subfamily of quadratic twists $\{E_f\}_{f \in F(n)(\F_{q})}$ is $p+1$ as $q \to \infty$, i.e.
\begin{equation*}
    \lim_{q \to \infty} \frac{\sum_{f \in F(n)(\F_{q})} |\Sel(E_f)|}{|F(n)(\F_{q})|} = p+1
\end{equation*}
\end{theorem}
\begin{proof}
As stated before, $|F(n)(\F_{q})| = q^{n+1} + O(q^n)$. As in the proof of theorem 3.13, the Grothendieck-Lefschetz trace formula and lemma 3.10 shows the following equation, where $\mathrm{Frob}_{q} \in \Gal(\bar{\F}_q/\F_{q})$. 
\begin{align*}
    \lim_{q \to \infty} \frac{\sum_{f \in F(n)(\F_{q})} |\Sel E_f|}{|F(n)(\F_{q})|} &= \lim_{q \to \infty} \frac{\sum_{d=1}^n \sum_{f \in F_d(\F_{q})} |\Sel E_f|}{\sum_{d=1}^n |F_d(\F_{q})|} \\
    &= \lim_{q \to \infty} \frac{\sum_{d=1}^n \sum_{f \in F_d(\F_{q})} |\Sel E_f|}{\sum_{d=1}^n |F_d(\F_{q})|} \\
    &= \lim_{q \to \infty} \frac{\sum_{d=1}^n \sum_{f \in F_d(\F_{q})} |\h^1_{\et}(C_{\bar{\F}_q}, \mathcal{E}_{f,p})^{\Gal(\bar{\mathbb{F}}_q / \F_{q})}|}{\sum_{d=1}^n |F_d(\F_{q})|} \\
    &= \lim_{q \to \infty} \frac{\sum_{d=1}^n |\tau_{d,p,E}(\F_{q})|}{\sum_{d=1}^n |F_d(\F_{q})|} 
    \left( = \lim_{q \to \infty} \frac{|\tilde{\tau}(n,p,E)(\F_{q})|}{|F(n)(\F_{q})|} \right) \\
    &= \lim_{q \to \infty} \sum_{d=1}^n \left( \frac{|\tau_{d,p,E}(\F_{q})|}{|F_d(\F_{q})|} \frac{|F_d(\F_{q})|}{|F(n)(\F_{q})|} \right) 
\end{align*}
By theorem 3.13, the following equation holds.
\begin{align*}
    \lim_{q \to \infty} \frac{\sum_{f \in F(n)(\F_{q})} |\Sel E_f|}{|F(n)(\F_{q})|} &= \lim_{q \to \infty} \sum_{d=1}^n \left( (p+1) \frac{|F_d(\F_{q})|}{|F(n)(\F_{q})|} \right) \\
    &= \lim_{q \to \infty} (p+1) \frac{|F_n(\F_{q})|}{|F(n)(\F_{q})|} \\
    &= p+1
\end{align*}
\end{proof}

We now order the family of quadratic twist of elliptic curves based on the canonical height of the elliptic curve. Recall that the canonical height of the elliptic curve is given as follows, where $E' : y^2 = x^3 + C(t)x + D(t)$ is an elliptic curve isomorphic to $E$.
\begin{equation*}
    h(E) := \text{inf}_{E' \cong E} \left( \text{max} \{ 3 \deg C, 2 \deg D \} \right)
\end{equation*}
\begin{remark}
There is a unique equation for $E$ of the form $y^2=x^3+Ax+B$ satisfying that for any prime $p \in \F_q[t]$, $p^4|A$ implies $p^6 \nmid B$. In such case, $h(E)$ is equal to $\ \text{max} \{ 3 \deg A, 2 \deg B \}$.

In particular $E $ has the least height among all quadratic twists if and only if for any prime $p \in \F_q[t]$, $p^2|A$ implies $p^3 \nmid B$. Otherwise, there exists a quadratic twist $E_p: py^2=x^3+Ax+B \simeq y^2= x^3+ \frac{A}{p^2}x + \frac{B}{p^3}$, which has smaller height than $E$.
\end{remark}

\begin{remark}
In order to apply the construction of the \'{e}tale $\F_p$-lisse sheaf $\tau_{n,p,E} \rightarrow F_n$ from \cite{CHall}, we need the assumption that the quadratic twist family we start with has members whose \neron model admits at least one multiplicative reduction away from $\infty$.

This is essentially necessary because one can find families of quadratic twists of some given elliptic curves, such that the whole family has no elliptic curve whose \neron model has multiplicative reductions. For instance, suppose $4$ and $27$ are invertible over the field $\F_q$. Then there are $\pi_1,\pi_2 \in \F_q[t]$ such that $4\pi_1t^3 + 27 \pi_2(t+1)^2=1$ with $deg(\pi_1)=1$ and $deg(\pi_2)=2$. One can readily check that the elliptic curve $E: y^2=x^3+\pi_1\pi_2t+\pi_1\pi_2^2(t+1)$ has discriminant $\Delta_E=\pi_1^2\pi_2^3(4\pi_1t^3+27\pi_2(t+1)^2)= \pi_1^2\pi_2^3$. Therefore, the \neron model of $E$ has no multiplicative reduction by Tate's algorithm. This elliptic curve $E$ has the least height in the quadratic twist family by remark 4.5. So, the whole quadratic twist family has no member whose \neron model admits multiplicative reductions away from $\infty$ by remark 3.4.

\end{remark}

Now we prove the main theorem.
\begin{proof}[Proof of Theorem 1.1]

Let $E: y^2 = x^3 + A_0 x + B_0$ be any non-isotrivial elliptic curve over $\F_q[t]$. We can always replace $E$ by $E_0 \in \{E_f\}$ that has minimal canonical height among all quadratic twists. Since we assumed that there exists at least one quadratic twist of $E$ whose \neron model admits a multiplicative reduction away from $\infty$, $E$ admits at least one multiplicative reduction.

\medskip
\textbf{Setup}
\medskip

Choose large enough $q$ such that the discriminant of $E / \F_{q}[t]$, denoted by $\Delta_E$, splits completely into linear factors as follows.
\begin{equation*}
    \Delta_E = \pi_0^{r_0} \pi_1^{r_1} \cdots \pi_m^{r_m}
\end{equation*}
Without loss of generality, assume E has multiplicative reduction at $\pi_0$. Note we can guarantee that the primes $\pi_i$'s are all linear. For those large enough $q$, we will explicitly determine the collection of all possible quadratic twists of $E / \F_{q}[t]$ whose height is bounded by $n$.

We order the family of quadratic twists of $E$ by canonical height. For any elliptic curve $E / \F_{q}[t]$, there exists a unique way to write $E$ as $y^2 = x^3 + Ax + B$ such that for any irreducible polynomial $p \in \F_{q}[t]$, if $p^4 | A$, then $p^6 \nmid B$. Then the canonical height of $E$ is given as follows.
\begin{equation*}
    h(E) = \mathrm{max} \{3 \deg A, 2 \deg B\}
\end{equation*}
In particular, we chose $E$ to have the least height in the family of quadratic twists of $E_0$. Hence, the coefficients $A,B$ of $E$ satisfy the aforementioned condition.

Any quadratic twist of $E / \F_{q}[t]$ can be uniquely written as $E_f: fy^2 = x^3 + Ax + B$ such that $f \in \F_{q}[t]$ is square-free. Assume $f = \pi_0^a \pi_{i_1} \pi_{i_2} \cdots \pi_{i_s} g$ where $\pi_{i_j}$'s are distinct primes belong to the set $\{\pi_1, \cdots, \pi_m\}$, $a = 0$ or $1$, and $g$ is a square-free polynomial such that $(g, \Delta_E) = 1$ in $\bar{\F}_q[t]$. Denote by $J$ the subset $\{\pi_{i_1}, \pi_{i_2}, \cdots, \pi_{i_s}\}$ of $\{\pi_1, \pi_2, \cdots, \pi_m\}$.

Fix a positive integer $n$. We now consider the quadratic twists $\{E_f\}$ whose height $h(E_f) \leq n$. 

\medskip
\textbf{Case 1}
\medskip

Suppose that $a = 0$, i.e. $\pi_0$ is not a prime factor of $f$. Remark 4.5 implies that the twist $E_f$ is isomorphic to the following elliptic curve, which is a minimal model.
\begin{equation*}
    y^2 = x^3 + \left( \prod_{\pi_{i_j} \in J} \pi_{i_j}^2 \right) g^2 Ax + \left( \prod_{\pi_{i_j} \in J} \pi_{i_j}^3 \right) g^3 B \tag{*}
\end{equation*}

We will use use $(*)$ to explicitly compute the height of $E_f$. We also note that for any prime $p | g$, $v_p(A) =0$ or $v_p(B) =0$. Otherwise, $g$ is not coprime to $\Delta_E$. Therefore, we have the following equivalent relation.
\begin{equation*}
    h(E_f) \leq n \Longleftrightarrow \mathrm{max} \left\{ \deg \left( \left( \prod_{\pi_{i_j} \in J} \pi_{i_j}^6 \right)  g^6 A^3 \right), \deg \left( \left( \prod_{i_j \in J} \pi_{i_j}^6 \right) g^6 B^2 \right) \right\} \leq n
\end{equation*}
We set $M := \mathrm{max}\{3 \deg A, 2 \deg B\}$ and $M_J := M + 6 \deg \left( \prod_{i_j \in I} \pi_{i_j} \right) = M + 6 |J|$. Hence the following equivalent relation holds.
\begin{equation*}
    h(E_f) \leq n \Longleftrightarrow \deg g \leq \frac{n - M_J}{6}
\end{equation*}
It is crucial to notice that the twist $E_f$ still has at least one place of multiplicative reduction at $\pi_0$, which can be checked using Tate's algorithm.

Denote by $E_J$ the following elliptic curve, which is minimal by remark 4.5.
\begin{equation*}
    E_J : y^2 = x^3 + \left( \prod_{\pi_{i_j} \in J} \pi_{i_j}^2 \right) Ax + \left( \prod_{\pi_{i_j} \in J} \pi_{i_j}^3 \right) B
\end{equation*}
Then the following equation holds.
\begin{align*}
    \sum_{\substack{f = \pi_{i_1} \pi_{i_2} \cdots \pi_{i_s} g \\ (g,\Delta_E) = 1 \\  h(E_f) \leq n}} |\Sel E_f| &= \sum_{g \in F \left( \frac{n - M_{J}}{6}\right) (\F_{q})}  |\Sel(E_J)_g| \\
    &= \left| \tilde{\tau} \left( \frac{n - M_{J}}{6}, p, E_{J} \right) (\F_{q}) \right| \\
    &= (p+1) \left( q^{\frac{n - M_{J}}{6} + 1} + \mathrm{O}_{n,p}(q^{\frac{n-M_J}{6}}) \right)
\end{align*}

\medskip
\textbf{Case 2}

Now assume $a = 1$, i.e. $f = \pi_0 \pi_{i_1} \pi_{i_2} \cdots \pi_{i_s} g$ such that $(g, \Delta_E) = 1$ and $J = \{\pi_{i_1}, \pi_{i_2}, \cdots, \pi_{i_s}\}$ is a subset of $\{\pi_1, \pi_2, \cdots, \pi_m\}$. Define $M$ and $M_J$ analogously to Case 1. Then, by the aformentioned argument in Case 1, we have that $E_f$ is isomorphic to the following minimal model.
\begin{equation*}
    y^2 = x^3 + \left( \prod_{\pi_{i_j} \in J} \pi_{i_j}^2 \right) \pi_0^2 g^2 Ax + \left( \prod_{\pi_{i_j} \in J} \pi_{i_j}^3 \right) \pi_0^3 g^3 B \tag{*}
\end{equation*}
Thus the height of $E_f$ can be written as follows.
\begin{equation*}
    h(E_f) \leq n \Longleftrightarrow \deg g \leq \frac{n - M_{J} - 6 \deg \pi_0}{6} = \frac{n - M_{J}}{6} - 1
\end{equation*}
We will use the bound of $h(E_f)$ to estimate the summation of the size of the $p$-Selmer group for quadratic twists $E_f$ by using lemma 3.7. Corollary 3.3 and lemma 3.7 implies that the maximal size of $p$-Selmer group of $E_f$ is the following.
\begin{equation*}
    |\Sel E_f| \leq p^{2 \deg \Delta_{E_f} + 4} \leq p^{2 h(E_f) + 4} = p^{2n + 4}
\end{equation*}
Therefore, the following equation gives the approximation on the size of the $p$-Selmer group over $\{E_f\}$ for the desired collection of $f$.
\begin{align*}
    \sum_{\substack{f = \pi_0 \pi_{i_1} \pi_{i_2} \cdots \pi_{i_s} g \\ (g,\Delta_E) = 1 \\  h(E_f) \leq n}} |\Sel E_f| &= \sum_{g \in F \left( \frac{n - M_{J}}{6} - 1\right) (\F_{q})}  \Sel(E_{J})_g \\
    &= \mathcal{M} \left( q^{\frac{n - M_{J}}{6}} + \mathrm{O}_{n,p}(q^{\frac{n - M_J}{6} - 1}) \right)
\end{align*}
Here, $\mathcal{M}$ is a positive integer such that $1 \leq \mathcal{M} \leq p^{2n+4}$.

\medskip
\textbf{Average Size}
\medskip

Using both aforementioned cases, we can now calculate the average size of $p$-Selmer group as $k \to \infty$. Recall that $\mathbb{E}_{n,k,p}$ is the average value of $p$-Selmer groups over families of quadratic twists of canonical height at most $n$. Then the following equation holds for any fixed $n \geq 30$. Note that we may need to require $n \geq 30$ because of the conditions on the degree of twisting polynomials from lemma 3.12 and theorem 3.13.
\begin{align*}
    \lim_{q \to \infty} \mathbb{E}_{n,p} &= \lim_{q \to \infty} \frac{ \sum_{J \subset \{\pi_1, \cdots, \pi_m\}}(p+1) \left( q^{\frac{n - M_J}{6} + 1} + \mathrm{O}_{n,p}(q^{\frac{n-M_J}{6}}) \right) + \mathcal{M} \left( q^{\frac{n - M_J}{6}} + \mathrm{O}_{n,p}(q^{\frac{n-M_J}{6} - 1}) \right)}{\sum_{J \subset \{\pi_1, \cdots, \pi_m\}} |F(\frac{n - M_J}{6})(\F_{q})| + |F(\frac{n - M_J}{6}-1)(\F_{q})|} \\
    &= \lim_{q \to \infty} \sum_{J \subset \{\pi_1, \cdots, \pi_m\}} \frac{(p+1) \left( q^{\frac{n - M_J}{6} + 1} + \mathrm{O}_{n,p}(q^{\frac{n-M_J}{6}}) \right) + \mathcal{M} \left( q^{\frac{n - M_J}{6}} + \mathrm{O}_{n,p}(q^{\frac{n-M_J}{6}-1}) \right)}{|F(\frac{n - M_J}{6})(\F_{q})| + |F(\frac{n - M_J}{6}-1)(\F_{q})|} \\ 
    &\times \lim_{q \to \infty}  \frac{|F(\frac{n - M_J}{6})(\F_{q})| + |F(\frac{n - M_J}{6}-1)(\F_{q})|}{\sum_{J \subset \{\pi_1, \cdots, \pi_m\}} |F(\frac{n - M_J}{6})(\F_{q})| + |F(\frac{n - M_J}{6}-1)(\F_{q})|} \\
    &= \lim_{q \to \infty} \sum_{J \subset \{\pi_1, \cdots, \pi_m\}} \frac{(p+1) \left( q^{\frac{n - M_J}{6} + 1} + \mathrm{O}_{n,p}(q^{\frac{n-M_J}{6}}) \right) + \mathcal{M} \left( q^{\frac{n - M_J}{6}} + \mathrm{O}_{n,p}(q^{\frac{n-M_J}{6}-1}) \right)}{q^{\frac{n - M_J}{6} + 1} + \mathrm{O}_{n,p}(q^{\frac{n-M_J}{6}}) + q^{\frac{n - M_J}{6}} + \mathrm{O}_{n,p}(q^{\frac{n-M_J}{6}-1})} \\
    &\times \lim_{q \to \infty} \frac{|F(\frac{n - M_J}{6})(\F_{q})| + |F(\frac{n - M_J}{6}-1)(\F_{q})|}{\sum_{J \subset \{\pi_1, \cdots, \pi_m\}} |F(\frac{n - M_J}{6})(\F_{q})| + |F(\frac{n - M_J}{6}-1)(\F_{q})|} \\
    &= \lim_{q \to \infty} \sum_{J \subset \{\pi_1, \cdots, \pi_m\}} (p+1) \frac{|F(\frac{n - M_J}{6})(\F_{q})| + |F(\frac{n - M_J}{6}-1)(\F_{q})|}{\sum_{J \subset \{\pi_1, \cdots, \pi_m\}} |F(\frac{n - M_J}{6})(\F_{q})| + |F(\frac{n - M_J}{6}-1)(\F_{q})|} \\
    &= (p+1) \lim_{q \to \infty} \sum_{J \subset \{\pi_1, \cdots, \pi_m\}} \frac{|F(\frac{n - M_J}{6})(\F_{q})| + |F(\frac{n - M_J}{6}-1)(\F_{q})|}{\sum_{J \subset \{\pi_1, \cdots, \pi_m\}} |F(\frac{n - M_J}{6})(\F_{q})| + |F(\frac{n - M_J}{6}-1)(\F_{q})|} \\
    &= p+1
\end{align*}
Therefore, the average size of $p$-Selmer groups of family of quadratic twists of any elliptic curve $E$ over $\F_q[t]$ is given by $p+1$:
\begin{equation*}
    \lim_{n \to \infty} \lim_{q \to \infty} \mathbb{E}_{n,p} = p+1
\end{equation*}
\end{proof}
\nocite{*}
\bibliographystyle{alpha}
\bibliography{chmix}

\printindex

%\bibliography{chmix}
%\bibliographystyle{amsrefs}
%\bibliography{ch-mix}
%\bibliography{journals_abbrev}
\end{document}